\documentclass[12pt]{amsart}
\usepackage{euscript}
\usepackage{amssymb}
\usepackage{latexsym}
\usepackage{amsmath}

\def\sD{{\mathfrak D}}      
   \def\sH{{\mathfrak H}}   
      
\def\sM{{\mathfrak M}}   \def\sN{{\mathfrak N}}

      \def\dC{{\mathbb C}}

\def\cA{{\mathcal A}}      
\def\cD{{\mathcal D}}      
\def\cG{{\mathcal G}}   \def\cH{{\mathcal H}}   
      \def\cL{{\mathcal L}}
      
      \def\cR{{\mathcal R}}
      \def\cU{{\mathcal U}}
\def\cV{{\mathcal V}}

   \def\bB{{\mathbf B}}

\def\ran{{\rm ran\,}}
\def\cran{{\rm \overline{ran}\,}}
\def\dom{{\rm dom\,}}

\def\codim{{\rm codim\,}}

\def\uphar{{\upharpoonright\,}}

\def\f{\varphi}

\def\half{{\frac{1}{2}}}

\newtheorem{theorem}{Theorem}[section]
\newtheorem{lemma}[theorem]{Lemma}
\newtheorem{proposition}[theorem]{Proposition}
\newtheorem{corollary}[theorem]{Corollary}

\newtheorem{remark}[theorem]{Remark}

\numberwithin{equation}{section}

\def\IM{{\rm Im\,}}

\def\wt{\widetilde}
\def\wh{\widehat}

\usepackage{color}

\begin{document}
\title
[$C$-self-adjoint contractive extensions]
{$C$-self-adjoint contractive extensions of $C$-symmetric non-densely defined contractions}
\author[Yury Arlinski\u{\i}]{Yury Arlinski\u{\i}}
\author[Konrad Schmüdgen]{Konrad Schmüdgen}
\address{Volodymyr Dahl East Ukrainian National University, Kyiv, Ukraine} \email{yury.arlinskii@gmail.com}
\address{University of Leipzig, Mathematical Institute, Augustusplatz 10/11, D-04109 Leipzig, Germany}
\email{schmuedgen@math.uni-leipzig.de}
\begin{abstract}
Given  a conjugation (involution) $C$ on a Hilbert space, $C$-self-adjoint contractive extensions of a non-densely defined $C$-symmetric contraction are studied and
 parametrizations of all such extensions are obtained. As an  application, a proof of an announced result of Glazman \cite{Glazman1} on the existence of maximal
dissipative $C$-self-adjoint extensions of a densely defined $C$-symmetric dissipative operator is given.

\end{abstract}

\maketitle
\thispagestyle{empty} \textbf{Key words:} contraction, conjugation, dissipative operator, complex symmetric operator

\maketitle
\textbf{AMS  Subject  Classification (2020)}.
 47A20, 47B44, 47B02.\\

\tableofcontents
\section{Introduction}
A  conjugation (or involution) on a Hilbert space $\cH$ is an anti-linear isometric operator on $\cH$ such that $C^2=I_\cH$, see \cite{Glazman1,GP1, GP2,
Schmudgen, Mashreghi}.
Then it follows that
\begin{equation}\label{jcyjdf}
(Cf,g)=(Cg, f)\;\;\forall f,g\in\cH.
\end{equation}
If $C$ be a conjugation on $\cH$, a linear operator $B$ on $\cH$ is called $C$-symmetric if
\begin{equation}\label{kvit23aa}
(Bx, Cy)=(x, CBy)\;\; \forall x,y\in\dom B.
\end{equation}
If $B$ is densely defined, then \eqref{kvit23aa}
is equivalent to $B\subseteq CB^*C$. If $B= CB^*C$, then the operator $B$ is called $C$-self-adjoint.
These notions were introduced by Glazman \cite{Glazman3}.

The present paper deals with norm preserving $C$-self-adjoint extensions of  non-densely defined bounded $C$-symmetric operators.

The study of norm preserving extensions of linear operators and their interplay with other problems has a long history in operator theory,
 see e.g. \cite{ArlTsek1986,Arl1999,AG1982,DKW,FoFra1984, Kr, Parrot, Phillips1959, Schmudgen, ShYa, SF}. A sample is von Neumann's theory of the Cayley
 transform which connects self-adjoint extensions of a symmetric operator with unitary  extensions of a non-densely defined isometric operator \cite{AG,
 Schmudgen}. The Cayley transform reduces also the problem of maximal dissipative (accretive) extensions of a  dissipative (accretive) operator to
 contractive extensions of a non-densely defined contraction.

The starting point and first motivation for this paper was the following result which was announced by I.M. Glazman in \cite{Glazman1}:
\begin{theorem}\label{th1}
\textit{Each $C$-symmetric ($J$-symmetric in Glazman's terminology) dissipative operator with dense domain admits a $C$-self-adjoint maximal dissipative
extension}.
\end{theorem}
According to   our knowledge, a proof of this theorem never appeared. In \cite{Glazman1} it is argued as follows:  Suppose that $T$ is a closed densely
defined dissipative operator and let $V=(T-iI)(T+iI)^{-1}$ be its Cayley transform.
Then $V$ is a contraction defined on the subspace $\cH_V:=\ran (T+i I)$. If $T$ is  $C$-symmetric, then $V$ is $C$-symmetric as well.
Further, it is noted in \cite{Glazman1} that a modification of Krein's approach in \cite{Kr} gives the existence of contractive $C$-self-adjoint extension
$\wt V$ of $V$. Having $\wt V$ the inverse Cayley transform $\wt T=i(I+\wt V)(I-\wt V)^{-1}$, $\dom \wt T=\ran(I-\wt V),$
yields a $C$-self-adjoint maximal dissipative extension of $T$.

A proof of the existence of such an extension $\wt V$ was missing.
In Theorem \ref{abe17a} of this paper we prove the existence of $C$-self-adjoint contractive extensions of  non-densely defined $C$-symmetric contractions.
In particular, combined with the reasoning of the preceding paragraph, this completes the proof of   Glazman's Theorem \ref{th1}.

Let us sketch the main results of this paper. In Theorem \ref{kvit20ab}  we describe all  $C$-self-adjoint bounded extensions of a non-densely defined
$C$-symmetric bounded operator.
A description of these extensions has been already established in \cite[Theorem 1]{Raikh1975}. Our Theorem \ref{kvit20ab} gives a new form of such
extensions.

Theorems \ref{abe17a} and  \ref{kbg08ab} contain  complete descriptions of all contractive $C$-self-adjoint extensions of a non-densely defined $C$-symmetric contraction and $C$-self-adjoint unitary extensions of a $C$-symmetric non-densely defined isometry.  For
this we use Crandall's parameterization \cite{Crandal} of  contractive extensions of  non-densely defined contractions.

A particular emphasis in this paper is on the uniqueness of the corresponding extensions. In Theorems  \ref{kkvit23ab} and  \ref{kbg08ab}, we characterize the uniqueness of
$C$-self-adjoint contractive extensions and $C$-self-adjoint unitary extensions, respectively.  Theorems \ref{ber22aa} and \ref{kbh10aa} deal with  the uniqueness of $C$-self-adjoint maximal dissipative extensions and self-adjoint extensions.

Let us fix  some notation. The Banach space of all bounded operators acting between Hilbert spaces $\cH_1$ and $\cH_2$ is denoted by $\bB(\cH_1,\cH_2)$ and
$\bB(\cH):=\bB(\cH,\cH)$.
The symbols $\dom T$, $\ran T$, $\ker T$ denote
the domain, range and kernel of a linear operator $T$, respectively,
and $\cran T$ denotes the closure of the range of $T$. If $Z\in \bB(\cH_1,\cH_2)$ is a contraction,  we  abbreviate
\begin{equation*}\cD_Z:=(I-Z^*Z)^\half \quad \textrm{and}\quad \sD_Z:=\cran \cD_Z.
\end{equation*}

\section{Bounded $C$-self-adjoint extensions of a non-densely defined $C$-symmetric bounded operator}
Let us begin with two simple facts which are of interest in itself.
\begin{lemma}\label{lip09ad}
Let $C$ be a conjugation in the Hilbert space $\cH$. Then the map $J\to U:=CJ$ is a bijection of the set of conjugations $J$ on the set of $C$-self-adjoint unitary operators $U$.
\end{lemma}
\begin{proof}
If $J$ is a conjugation, then $U=CJ$ is obviously a unitary and $C=UJ$, so  $CUC=CCJ C=JC=U^*$, that is, $U$ is $C$-self-adjoint.

Conversely, let $U$ be a $C$-self-adjoint unitary. Set $J:=CU$. Then $J$ is anti-linear, isometric, $U=CJ$ and $J^2=CUCU=U^*U=I$, that is, $J$ is a conjugation.
\end{proof}
\begin{lemma}\label{kvit20a}
Let $C$ be a conjugation in  $\cH$ and let $A\in \bB(\cH).$ Then the operator
\begin{equation}\label{kvit20aa}
\cA:=\half(A+CA^*C)
\end{equation}
is $C$-self-adjoint and, moreover, $||\cA||\le ||A||.$ If $A$ is $C$-self-adjoint, then  $\cA=A=CA^*C$.
\end{lemma}
\begin{proof}
Since $\cA^*=\half(A^*+CAC)$ we get $C\cA^*C=\half(CA^*C+A)=
\cA$, i.e., the operator $\cA$ is $C$-self-adjoint. Since $||CA^*C||\le ||A^*||=||A||$, we obtain $||\cA||\le ||A||$.
\end{proof}

Note that the $C$-self-adjointness of the operator $\cA$ defined in \eqref{kvit20aa} has been used in the proof of Proposition 2.27 in \cite{BKMMW19}.

 Let $V$ be a bounded linear operator in $\cH$ defined on a subspace $\cH_V$. Suppose $V$ is a $C$-symmetric w.r.t.  $C$, i.e.,
\begin{equation}\label{abe17aa}
(Vf, Cg)=(f, CVg),\;\forall f,g\in \cH_V.
\end{equation}
We are interested in bounded $C$-self-adjoint extensions $\wt\cV$ on $\cH$, i.e., bounded linear operators $\wt\cV$ that
\[
\dom \wt\cV=\cH, \; \wt\cV\uphar\cH_V=V,\; \wt\cV=C\wt\cV^*C.
\]
Clearly, the following are equivalent:
\begin{enumerate}
\def\labelenumi{\rm (\roman{enumi})}
\item $\wt\cV$ is a bounded and $C$-self-adjoint extension of $V$;
\item $C\wt\cV^*C$ is a bounded $C$-self-adjoint extension of $V$.
\end{enumerate}
If $\cU$ is the  bounded operator  defined by
\[
\cH_{\cU}:=\dom\cU=C\dom \cV=C\cH_{\cV},\; \cU h=C\cV Ch,\; h\in\cH_{\cU},
\]
then the equality \eqref{kvit23aa} implies that the operator $\cU$ is $C$-symmetric and the pair $\{ \cV,\cU\}$ forms the \textit{adjoint (dual) pair} of operators, i.e. we have
\[
(\cV f, h)=(f,\cU h)\; \;\forall f\in\cH_{\cV},\;\forall h\in\cH_{\cU}.
\]
The adjoint $\wt\cV^*$ of any $C$-self-adjoint bounded extension $\wt\cV$ of a $C$-symmetric operator $\cV$ is also a $C$-self-adjoint extension of $\cU$.

Let $V^*\in\bB(\cH,\cH_V)$ be the adjoint of $V\in\bB(\cH_V,\cH)$. Then \eqref{abe17aa} is equivalent to the equality
\begin{equation}\label{abe19a}
P_{\cH_V}CVh=V^*Ch\;\;\forall h\in\cH_V.
\end{equation}
Set $\cH_V^\perp:=\cH\ominus\cH_V$. Note that the formula
\begin{equation}\label{hvf02a}
\wt\cV:=VP_{\cH_V}+\wt LP_{\cH_V^\perp}
\end{equation}
gives is a one-to-one correspondence between all operators $\wt L\in\bB(\cH_V^\perp,\cH)$ and all bounded extensions $\wt\cV$  of $V$.
The following result has been established by Raikh in \cite[Theorem 1]{Raikh1975}.
\begin{theorem}\label{abe27a}
The formula
\[
\wt\cV=VP_{\cH_V}+(CV^*C+S)P_{\cH_V^\perp}
\]
 gives a
bijective correspondence between the set of all $C$-self-adjoint extensions of a bounded non-densely defined $C$-symmetric operator
$V$ and the set of all $C$-self-adjoint operators $S\in\bB(\cH_V^\perp, CH_V^\perp)$.
\end{theorem}
Next we will develop a new description of $C$-self-adjoint bounded extensions.
For this purpose we define the following linear operator
\begin{equation}\label{vfh1a}
Z_0:Ch\mapsto P_{\cH_V^\perp}CVh,\; h\in\cH_V.
\end{equation}
Clearly, $Z_0\in\bB(C\cH_V,\cH_V^\perp)$, $||Z_0||\le ||V||,$ and for $Z_0^*\in\bB(\cH_V^\perp, C\cH_V)$ we have
\begin{equation}\label{hvf02aa}
P_{\cH_V}CZ^*_0g=V^*Cg,\; \; g\in\cH_V^\perp.
\end{equation}
If $\sM$ is a subspace then, see \cite{Raikh1975}, the equality
\begin{equation}\label{hvf19a}
P_{C\sM}=CP_\sM C
\end{equation}
holds. Hence from \eqref{vfh1a} and \eqref{hvf02aa} it follows that
\begin{equation}\label{hvf03bbc}
CV^*Cg=Z_0^*g\;\;\forall g\in\cH_V^\perp.
\end{equation}

\begin{proposition}\label{vfh1aa}
Let $V$ be a $C$-symmetric bounded operator defined on a subspace $\dom V=\cH_V$ and let $Z_0$ be the operator defined by \eqref{vfh1a}.
Then a bounded extension $\wt\cV$ of  $V$ is $C$-self-adjoint if and only if the operator $\wt L\in\bB(\cH_V^\perp,\cH)$ in
\eqref{hvf02a} satisfies the following conditions:
\begin{equation}\label{be403bb}
\wt L^*\uphar C\cH_V=Z_0,
\end{equation}
\begin{equation}\label{hvf03bc}
P_{\cH_V^\perp}C\wt Lf=\wt L^*Cf\;\;\forall f\in \cH_V^\perp.
\end{equation}
\end{proposition}
\begin{proof}
Assume that $\wt\cV_{\wt L}=VP_{\cH_V}+\wt L P_{\cH_V^\perp}$ is a bounded $C$-self-adjoint extension of $V$, i.e.,
\[
C\wt\cV_{\wt L}=\wt\cV^*_{\wt L}C\Longleftrightarrow C\left(VP_{\cH_V}+\wt LP_{\cH_V^\perp}\right)=\left(V^*+\wt L^*\right)C.
\]
The latter is equivalent to the equations
\begin{equation}\label{hvf03ab}
\left\{\begin{array}{l}
P_{\cH_V}CVP_{\cH_V}+P_{\cH_V}C\wt LP_{\cH_V^\perp}=V^*CP_{\cH_V}+ V^*CP_{\cH_V^\perp},\\[2mm]
P_{\cH_V^\perp}CVP_{\cH_V}+P_{\cH_V^\perp}C\wt LP_{\cH_V^\perp}=\wt L^*C.
\end{array}\right.
\end{equation}
From \eqref{abe19a} it follows the equality
$
P_{\cH_V}C\wt LP_{\cH_V^\perp}=V^*CP_{\cH_V^\perp}.
$ 
Passing to the adjoint we obtain \eqref{be403bb}. Then the second equality in \eqref{hvf03ab} yields \eqref{hvf03bc}.
\end{proof}
\begin{corollary}
If $\codim \cH_V=1$ and $\wt L\in\bB(\cH_V^\perp,\cH)$ satisfies \eqref{be403bb}, then the operator $\wt\cV$ of the form \eqref{hvf02a} is a $C$-self-adjoint
extension of $V$.
\end{corollary}
\begin{proof}
Using \eqref{jcyjdf} we check the equality \eqref{hvf03bc} for any $\wt L\in\bB(\cH_V^\perp, \cH)$.  Since $\codim \cH_V=1$, we can write $\cH_V^\perp=\dC e_0$, where
$e_0\in \cH_V^\perp$, $||e_0||=1$. Then
\[
\begin{array}{l}
\wt L^*Ce_0=(\wt L^*Ce_0,e_0)e_0
=(Ce_0,\wt Le_0)e_0=(C\wt Le_0, e_0)e_0\\
=(P_{\cH_V^\perp} C\wt L e_0,e_0)e_0=P_{\cH_V^\perp} C\wt Le_0.
\end{array}.
\]
Hence \eqref{hvf03bc} is valid.  If, in addition, $\wt L^*\in\bB(\cH,\cH_V^\perp)$ is an extension of $Z_0$, then the equality \eqref{hvf03bc} holds as
well. Therefore $\wt\cV=VP_{\cH_V}+\wt L P_{\cH_V^\perp}$ is a $C$-self-adjoint extension of $V$.
\end{proof}

\begin{theorem}\label{kvit20ab}
\begin{enumerate}
\item[(1)] Each bounded $C$-slfadjoint extension $\wt T$ of $V$ has the form
\begin{equation}\label{vfh01b}
\wt T_{\wt Z}=\half\left(\wt\cV_{\wt Z}+C\wt\cV^*_{\wt Z}C\right),
\end{equation}
where $\wt\cV_{\wt Z}$ is given by \eqref{hvf02a} and $\wt Z=\wt L^*\in\bB(\cH,\cH_V^\perp)$ is an extension on $\cH$ of the operator $Z_0$ defined in
\eqref{vfh1a}. Conversely, each  extension $\wt Z\in\bB(\cH,\cH_V^\perp)$ of the operator $Z_0$ determines a $C$-self-adjoint  bounded extension
$\wt T_{\wt Z}$ of $V$ of the form \eqref{vfh01b}.
\item[(2)] The equalities
\begin{equation}\label{vfh04caa}
VP_{\cH_V}+CV^*CP_{\cH_V^\perp}=\wt T_{\wt Z_0}=\half (\wt\cV_{\wt Z_0}+C\wt\cV^*_{\wt Z_0}C),
\end{equation}
hold, where
\begin{equation}\label{hvf04tt}
\wt Z_0 =Z_0P_{C\cH_V};
\end{equation}
\item[(3)] If $\wt Z=Z_0P_{C\cH_V}+\wt YP_{C\cH_V^\perp}$, where $\wt Y\in\bB(C\cH_V^\perp,\cH_V^\perp)$, then
\begin{equation}\label{hvf04ss}
\wt T_{\wt Z}=\wt T_{\wt Z_0}+\half(\wt Y^*+C\wt YC)P_{\cH_V^\perp}.
\end{equation}
\end{enumerate}
\end{theorem}
\begin{proof}

(1): By Lemma \ref{kvit20a}, the operator $\wt T_{\wt Z}$ is $C$ self-adjoint. If $\wt\cV$ is a $C$-self-adjoint extension of $V$, then the equality
$\wt\cV=C\wt\cV^*C$ implies that $C\wt\cV^*C\uphar \cH_V=V$. Using \eqref{hvf02a}, for $h\in\cH_V$ we obtain
\[
\begin{array}{l}
Vh=C\wt\cV^*Ch=C(V^*+\wt L^*)Ch=CV^*Ch+C\wt L^*Ch\\
=CP_{\cH_V}CVh+C\wt L^*Ch =CCVh-CP_{\cH_V^\perp}CVh+C\wt L^*Ch=Vh-CP_{\cH_V^\perp}CVh+C\wt L^*Ch.
\end{array}
\]

Hence $CP_{\cH_V^\perp}CVh+C\wt L^*Ch$, i.e., $\wt L^*Ch=P_{\cH_V^\perp}CVh,$ $\forall h\in\cH_V$. Thus $\wt L^*\uphar C\cH_V=Z_0.$
Set
$$\wt Z:=\wt L^*\in\bB(\cH,\cH_V^\perp),\; \wt Y:=\wt L^*\uphar C\cH_V^\perp\in\bB(C\cH_V^\perp,\cH_V^\perp).$$ Then
\begin{equation}\label{hvf04a}
\left\{\begin{array}{l}
\wt Z=\wt L^*=Z_0P_{C\cH_V}+\wt YP_{C\cH_V^\perp}\\
\wt Z^*=Z_0^*+\wt Y^*\\
\wt\cV=\wt\cV_{\wt Z}=VP_{\cH_V}+\wt Z^*P_{\cH_V^\perp}=VP_{\cH_V}+(Z_0^*+\wt Y^*)P_{\cH_V^\perp}\\
\wt\cV^*=\wt\cV^*_{\wt Z}=V^*+\wt Z=V^*+ Z_0P_{C\cH_V}+\wt YP_{C\cH_V^\perp}
\end{array}\right..
\end{equation}
Since $\wt\cV=C\wt\cV^*C$, we get $\wt\cV=\half\left(\wt\cV_{\wt Z}+C\wt\cV^*_{\wt Z}C\right).$

Now suppose that $\wt Z\in\bB(\cH,\cH_V^\perp)$ is an extension of $Z_0$, $\wt\cV_{\wt Z}=VP_{\cH_V}+\wt Z^*P_{\cH_V^\perp}$ and
$\wt T_{\wt Z}=\half\left(\wt\cV_{\wt Z}+C\wt\cV^*_{\wt Z}C\right)$. Then $C\wt T^*_{\wt Z}C=\wt T_{\wt Z}$
and due to \eqref{abe19a} and \eqref{vfh1a} we have $T_{\wt Z}\uphar\cH_V=V$.

(2) and (3): Using  \eqref{hvf03bbc}, \eqref{hvf04a}, and since $C\wt\cV^*_{\wt Z}CP_{\cH_V}=VP_{\cH_V}$, $P_{C\cH_V} CP_{\cH_V^\perp}=0$, we
derive
\[
\begin{array}{l}
\wt\cV_{\wt Z}+C\wt\cV^*_{\wt Z} C=VP_{\cH_V}+\wt Z^*P_{\cH_V^\perp}+C\wt\cV^*_{\wt Z}CP_{\cH_V}+C\wt\cV^*_{\wt Z}CP_{\cH_V^\perp}\\
= VP_{\cH_V}+(Z_0^*+\wt Y^*)P_{\cH_V^\perp} +VP_{\cH_V}+C\left(V^*+Z_0P_{C\cH_V}
+\wt YP_{C\cH_V^\perp}\right)CP_{\cH_V^\perp}\\
=2VP_{\cH_V}+2 CV^*CP_{\cH_V^\perp}+(\wt Y^*+C\wt YC)P_{\cH_V^\perp}.
\end{array}
\]
Thus,
\[
\wt T_{\wt Z}=\half\left(\wt\cV_{\wt Z}+C\wt\cV^*_{\wt Z}C\right)=VP_{\cH_V}+ CV^*CP_{\cH_V^\perp}+\half(\wt Y^*+C\wt YC)P_{\cH_V^\perp}.
\]
If $\wt Y=0\in\bB(C\cH_V^\perp,\cH_V^\perp)$, i.e., $\wt Z=\wt Z_0$, see \eqref{hvf04tt}, then
equality \eqref{vfh04caa} holds. Therefore, for $\wt Z=Z_0P_{C\cH_V}+\wt YP_{C\cH_V^\perp}$ the representation \eqref{hvf04ss} of $\wt T_{\wt Z}$
is valid.
\end{proof}

\section{Contractive $C$-self-adjoint extensions of a non-densely $C$-symmetric contraction}\label{kbg08a}
A special case is that $V$ is a $C$ symmetric  non-densely defined contraction  and $\wt V$ is a $C$-self-adjoint contractive extension of $V$.

The following result belongs to M.G. Crandall \cite{Crandal}.
\begin{theorem}\label{abe21ba}
The formula
\begin{equation}\label{abe18a}
\wt V_{\wt K}:=VP_{\cH_V}+\cD_{V^*}\wt KP_{\cH_V^\perp}
\end{equation}
gives a bijective correspondence between all contractive extensions $\wt V$ on $\cH$ of a non-densely defined contraction $V$ and all contractions
$\wt K\in\bB(\cH_V^\perp,\sD_{V^*})$, where $V^*\in\bB(\cH,\cH_V)$ is the adjoint of $V\in\bB(\cH_V,\cH)$ and $\cD_{V^*}:=(I-VV^*)^\half$,
$\sD_{V^*}:=\cran\cD_{V^*}.$
\end{theorem}
Note that the adjoint $\wt V^*_{\wt K}$ of $\wt V_{\wt K}$  is of the form
\[
\wt V^*_{\wt K}=V^*+\wt K^*\cD_{V^*}
\]
and
\begin{equation}\label{abe23a}
\left\{\begin{array}{l}
||\cD_{\wt V_{\wt K}}f||^2=||\cD_VP_{\cH_V}f-V^*\wt KP_{\cH_V^\perp} f||^2+||\cD_{\wt K}P_{\cH_V^\perp}f||^2,\\[2mm]
||\cD_{\wt V^*_{\wt K}}f||^2=||\cD_{\wt K^*}\cD_{V^*}f||^2,\; f\in\cH.
\end{array}\right.
\end{equation}
\begin{proposition}\label{abe18aa}
A contractive extension $\wt V$ of a $C$-symmetric contraction $V$ is $C$-self-adjoint if and only if the contraction $\wt K\in\bB(\cH_V^\perp,\sD_{V^*})$ in
\eqref{abe18a} satisfies the following conditions:
\begin{equation}\label{abe18bb}
\wt K^*\cD_{V^*} Ch=P_{\cH_V^\perp}CVh\;\;
\forall h\in \cH_V,
\end{equation}
\begin{equation}\label{abe18bc}
P_{\cH_V^\perp}C\cD_{V^*}\wt Kf=\wt K^*\cD_{V^*}Cf\;\;\forall f\in \cH_V^\perp.
\end{equation}
\end{proposition}
\begin{proof}
Assume that $\wt V_{\wt K}=VP_{\cH_V}+\cD_{V^*}\wt KP_{\cH_V^\perp}$ is a $C$-self-adjoint contractive extension of $V$, i.e.,
\[
C\wt V_{\wt K}=\wt V^*_{\wt K}C\Longleftrightarrow C\left(VP_{\cH_V}+\cD_{V^*}\wt KP_{\cH_V^\perp}\right)=\left(V^*+\wt K^*\cD_{V^*}\right)C.
\]
The latter is equivalent to the equalities
\begin{equation}\label{abe19ab}
\left\{\begin{array}{l}
P_{\cH_V}CVP_{\cH_V}+P_{\cH_V}C\cD_{V^*}\wt KP_{\cH_V^\perp}=V^*CP_{\cH_V}+ V^*CP_{\cH_V^\perp},\\[2mm]
P_{\cH_V^\perp}CVP_{\cH_V}+P_{\cH_V^\perp}C\cD_{V^*}\wt KP_{\cH_V^\perp}=\wt K^*\cD_{V^*}C.
\end{array}\right.
\end{equation}
From \eqref{abe19a} it follows that
\[
P_{\cH_V}C\cD_{V^*}\wt KP_{\cH_V^\perp}=V^*CP_{\cH_V^\perp}.
\]
Passing to the adjoint yields \eqref{abe18bb} and \eqref{abe18bc} follows from
 the second equality in \eqref{abe19ab}.
\end{proof}
Note that if $\wt V$ is a $C$-self-adjoint contractive extension of $V$, then
\begin{equation}\label{abe23aa}
C\cD_{\wt V^*}^2C=\cD^2_{\wt V}.
\end{equation}Moreover, as  proved in \cite[Lemma 3.1]{Messirdi},    for all continuous functions $\f$ on $[0,1]$ the equality
\[
\f(\wt V^*\wt V)=C\f(\wt V\wt V^*)C
\]
holds.

Let $\dot{X_0}$ denote the mapping
\begin{equation}\label{abe17b}
\dot{X_0}:\cD_{V^*}Ch\mapsto P_{\cH_V^\perp}CVh,\; h\in\cH_V.
\end{equation}
Taking  the equality \eqref{abe19a} into account, for each $h\in\cH_V$ we have
\[\begin{array}{l}
||\cD_{V^*}Ch||^2-|| P_{\cH_V^\perp}CVh||^2=||Ch||^2-||V^*Ch||^2-|| P_{\cH_V^\perp}CVh||^2\\
=||h||^2-||P_{\cH_V}CVh||^2-|| P_{\cH_V^\perp}CVh||^2=||h||^2-||CVh||^2=||h||^2-||Vh||^2\ge 0.
\end{array}
\]
It follows that the equality $\cD_{V^*}Ch=0$ ($h\in\cH_V$) yields  $P_{\cH_V^\perp}CVh=0$. Therefore $\dot{X_0}$ is a linear operator and, moreover, $\dot{X_0}$ is a contraction. Besides, $\dot{X_0}$ is an isometry if and only if $V$ is isometry.

Hence $\dot{X_0}$ admits a continuation to a contraction in $\bB(\overline{\cD_{V^*}C\cH_V},\cH_V^\perp)$. We will use the notation $X_0$ for this continuation.

\begin{proposition}\label{bkm29a}
(1) If $\dom X_0=\{0\}$, then $V$ is isometry, $CV\cH_V\subseteq \cH_V$ and
\begin{equation}\label{bkm29ab}
VCVh=Ch\;\;\forall h\in\cH_V.
\end{equation}
(2) If \eqref{bkm29ab} holds, then $V$ is an isometry, $C\cH_V\subseteq\ran V$ and, therefore, $\dom X_0=\{0\}.$
Moreover, the equalities
\[
CV\cH_V=\cH_V,\;\;C\cH_V=\ran V
\]
hold.
\end{proposition}
\begin{proof}
(1) Assume $\dom X_0=\{0\}$, Then from \eqref{abe17b} we get the inclusion $C\cH_V\subseteq\ker\cD_{V^*}$ and $P_{\cH_V^\perp}CVh=0$ for all $h\in\cH_V$.
Thus, $VV^*Ch=Ch$ for all $h\in\cH_V$ and $CV\cH_V\subseteq\cH_V$. Now $V^*Ch=P_{\cH_V}CVh=CVh,$ see \eqref{abe19a}, implies \eqref{bkm29ab}. Since $V$ is a contraction and $C$ is conjugation, we get that $V$ is an isometry.

(2) As shown above, \eqref{bkm29ab} implies that $V$ is an isometry. Hence $V^*(VCVh)=V^*Ch$ $\Longrightarrow CVh=V^*Ch$, $h\in\cH_V$. 
Then, due to \eqref{bkm29ab}, we get
$$Ch=VCVh=VV^*Ch \Longrightarrow  Ch\in\ker \cD_{V^*}=\ran V\;\;\forall h\in\cH_V\Longrightarrow\dom X_0=\{0\}.$$

If $\f\in \cH_V$ is orthogonal to $CV\cH_V$, then
$V^*C\f\in \cH_V^\perp$. This inclusion contradicts to $\ran V^*\subseteq\cH_V$. Hence $V^*C\f=0$  and \eqref{bkm29ab} implies $CV\f=0$ $\Longrightarrow\f=0$.

Since $CVh=V^*Ch$ for all $h\in\cH_V$ and $V$ is an isometry, we get that the operator $V^*\uphar C\cH_V$ is isometric and $V^*C\cH_V=\cH_V (=\dom V)$. Hence
$C\cH_V=\ran V.$
The proof is complete.  
\end{proof}

Proposition \ref{abe18aa} and equality \eqref{abe18bb} say that if $\wt V_{\wt K}$ is a $C$-self-adjoint contractive extension of $V$, then the operator $\wt
K^*$ is a contractive extension of the operator $X_0$ on $\sD_{V^*}$.
All contractive extensions of the contraction $X_0$ can be described as follows.

Define the following subspaces of $\sD_{V^*}:$
\begin{equation}\label{cherv07b}
\begin{array}{l}
\cL_0:=\dom X_0=\overline{\cD_{V^*}C\cH_V},\\[2mm]
\cL^\perp_0:=\sD_{V^*}\ominus\cL_0=\{y\in\cD_{V^*}:\cD_{V^*}y\in C\cH_V^\perp\}=\cD_{V^*}^{-1}(C\cH_V^\perp).
\end{array}
\end{equation}

Clearly,
\[
\cL^\perp_0=\{0\}\Longleftrightarrow \ran \cD_{V^*}\cap C\cH_V^\perp=\{0\}.
\]
Due to Proposition \ref{bkm29a} the equality $\cL^\perp_0=\sD_{V^*}(\Longleftrightarrow\cL_0=\{0\})$ holds if and only if \eqref{bkm29ab} holds.

Let $X_0^*\in\bB(\cH_V^\perp, \cL_0 )$ be the adjoint of $X_0\in\bB(\cL_0 ,\cH_V^\perp)$. The operator $X_0^*$ is the adjoint of the operator
\[
\wt X_0:=X_0P_{\cL_0}
\]
as well. The equality \eqref{abe17b} implies that
\begin{equation}\label{vfh07a}
CP_{\cH_V}C\cD_{V^*}X_0^*P_{\cH_V^\perp}=CV^*CP_{\cH_V^\perp}.
\end{equation}
Observe that \eqref{vfh1a}, \eqref{hvf02aa}, \eqref{abe17b}, and \eqref{vfh07a} yield
\[
\begin{array}{l}
X_0P_{\cL_0}\cD_{V^*}Ch=Z_0Ch=P_{\cH_V^\perp}CVh\; \;h\in\cH_V\\
CV^*CP_{\cH_V^\perp}f=Z^*_0f=P_{C\cH_V}\cD_{V^*}X^*_0f=CP_{\cH_V}C\cD_{V^*}X^*_0f\\
=\cD_{V^*}X^*_0f-CP_{\cH_V^\perp}C\cD_{V^*}X^*_0f\;\;\forall f\in\cH_V^\perp.
\end{array}
\]
Therefore
\begin{equation}
\label{hvf08aab}
\cD_{V^*}X^*_0f=CV^*Cf+CP_{\cH_V^\perp}C\cD_{V^*}X^*_0f=CV^*Cf+P_{C\cH_V^\perp}\cD_{V^*}X^*_0f,\;\; f\in\cH_V^\perp.
\end{equation}
By Theorem \ref{abe21ba} all contractive extensions of $X_0$ on $\sD_{V^*}$ are of the form
\begin{equation}\label{hvf06a}
\wt X=X_0 P_{\cL_0}+ \cD_{X^*_0}\wt YP_{\cL_0^\perp},\; \quad\mbox{where} \quad \wt Y\in \bB(\cL_0^\perp,\sD_{X^*_0})\quad\mbox{is a contraction}.
\end{equation}

Our next aim is to prove the following theorem.
\begin{theorem}\label{abe17a}
Let $V$ be a $C$-symmetric contraction, $\dom V=\cH_V$. Then:
\begin{enumerate}
\item[(1)] Each $C$-self-adjoint contractive extension $\wt W$ of $V$ is of the form
\begin{equation}\label{abe21aa}
\wt W=\half\left(\wt V_{\wt K}+C\wt V^*_{\wt K}C\right), \; \wt K=\wt X^*,
\end{equation}
where
$\wt X\in\bB(\sD_{V^*},\cH_V^\perp)$ is a contractive extension of the contraction $X_0$, defined in \eqref{abe17b}, see \eqref{hvf06a}.
Conversely, each contractive extension $\wt X\in\bB(\sD_{V^*},\cH_V^\perp)$ of the contraction $X_0$ determines a $C$-self-adjoint contractive extension
$\wt W$ of $V$ of the form \eqref{abe21aa}.
\item[(2)] Setting
\begin{equation}\label{vfh07aad}
\wt W_0:=\half\left(\wt V_{X_0^*}+C\wt V^*_{ X_0^*}C\right),
\end{equation}
we have
\begin{equation}\label{vfh07abc}
\begin{array}{l}
\wt W_0=\wt V_{X^*_0}+\half\left(C\wt X_0D_{V^*}C-CP_{\cH_V^\perp}CD_{V^*}X_0^*  \right)P_{\cH_V^\perp}\\[2mm]
=VP_{\cH_V}+CV^*CP_{\cH_V^\perp}+\half\left(CP_{\cH_V^\perp}C\cD_{V^*}X^*_0+C\wt X_0 \cD_{V^*}C\right)P_{\cH_V^\perp},
\end{array}
\end{equation}
and
 the operator $\wt W$ in \eqref{abe21aa} admits the representation
\begin{equation}\label{vfh07cda}
\wt W=\wt W_0+\half\left(M(\wt Y)+CM^*(\wt Y)C\right)P_{\cH_V^\perp},
\end{equation}
where $\wt K^*$ is of the form \eqref{hvf06a} and
\[
M(\wt Y):=\cD_{V^*}\wt Y\cD_{X^*_0},\;\; \wt Y\in\bB(\sD_{X^*_0}, \cD^{-1}_{V^*}(C\cH_V^\perp))\quad\mbox{is a contraction};
\]

\end{enumerate}
\end{theorem}
\begin{proof} (1): If $\wt K=\wt X^*$, where $\wt X\in\bB(\sD_{V^*},\cH_V^\perp)$, is a contractive extension of the contraction $X_0$, then
$C\wt V_{\wt K}^*C$ as well as $\wt V_{\wt K}$ is a contractive extension of $V$.
Actually for $h\in\cH_V$, using \eqref{abe19a} and \eqref{abe17b}, we get
\[
\begin{array}{l}
C\wt V_{\wt K}^*Ch=C(V^*+\wt K^*\cD_{V^*})Ch=CV^*Ch+ C\wt K^*\cD_{V^*}Ch=CP_{\cH_V}CVh+CX_0(\cD_{V^*}Ch)\\[2mm]
=CP_{\cH_V}CVh+ P_{\cH_V^\perp}CVh=C(P_{\cH_V}CVh+ P_{\cH_V^\perp}CVh)=CCVh=Vh.
\end{array}
\]

(2):
Let the operator $\wt W$ be given by \eqref{abe21aa}.
Then $\wt W$ is a contraction, $\wt W\uphar \cH_V=V$, and
\[
C\wt W^*C=\half C\left(\wt V_{\wt K}^*+C\wt V_{\wt K}C\right)C=\half\left(\wt V_{\wt K}+C\wt V^*_{\wt K}C\right)=\wt W.
\]
Thus, $\wt W$ is a $C$-self-adjoint contractive extension of $V$.

If $\wt V$ is a $C$-self-adjoint contractive extensions of $V$, then $C\wt V^*C=\wt V$. Hence we have
$\wt V=\half(\wt V+C\wt V^*C)$. Since $\wt V=\wt V_{\wt K}$ for some $\wt K\in\bB(\cH_V^\perp,\sD_{V^*})$,  by Proposition \ref{abe18aa} and by the
equalities \eqref{abe18bb}, \eqref{abe17b}, the operator $\wt K^*\in\bB(\sD_{V^*},\cH_V^\perp)$ is a contractive extension of  $X_0$.

Since $C\wt V^*_{\wt K}C$ and $\wt W$ are contractive extensions of $V$, by Theorem \ref{abe21ba} the operators $C\wt V^*_{\wt K}C$ and $\wt W$ take the form
\begin{equation}\label{abe22ba}
C\wt V^*_{\wt K}C=VP_{\cH_V}+\cD_{V^*}\wt MP_{\cH_V^\perp},\;
\wt W=VP_{\cH_V}+\cD_{V^*}\wt LP_{\cH_V^\perp},
\end{equation}
where $\wt M,\wt L\in\bB(\cH_V^\perp,\sD_{V^*})$ are contractions and \eqref{abe21aa}, \eqref{abe22ba} give
\[
\cD_{V^*}\wt LP_{\cH_V^\perp}=\half (\cD_{V^*}\wt KP_{\cH_V^\perp}+\cD_{V^*}\wt MP_{\cH_V^\perp}).
\]
Hence
\begin{equation}\label{abe22a}
\wt L=\half (\wt K+\wt M)\in\bB(\cH_V^\perp,\sD_{V^*})\Longleftrightarrow \wt L^*=\half (\wt K^*+\wt M^*)\in\bB(\sD_{V^*},\cH_V^\perp).
\end{equation}
Because $\wt W$ is a C-self-adjoint contractive extension of $V$, by Proposition \ref{abe18aa} and by the equalities \eqref{abe18bb} and \eqref{abe17b}, the
operator $\wt L^*\in\bB(\sD_{V^*},\cH_V^\perp)$ is a contractive extension of  $X_0$.
Let
\[
\wt K=\wt X^*=X_0^* + \wt Y^*\cD_{X^*_0},
\]
where $\wt Y\in\bB(\cH_V^\perp,\cL^\perp_0)$ is a contraction.
 Since $\wt K^*$ is a contractive extension of $X_0$,
we obtain from \eqref{abe22a} that $\wt M^*$ is a contractive extension of $X_0$.

Then, using  \eqref{hvf06a} and \eqref{hvf08aab} we compute
\[
\begin{array}{l}
\wt V_{\wt K}=VP_{\cH_V}+D_{V^*}\wt K P_{\cH_V^\perp}=VP_{\cH_V}+\cD_{V^*}(X_0^*+\wt Y^*\cD_{X^*_0})P_{\cH_V^\perp}\\[2mm]
=VP_{\cH_V}+\cD_{V^*}X_0^*P_{\cH_V^\perp}+\cD_{V^*}\wt Y^*\cD_{X^*_0}P_{\cH_V^\perp}\\
=VP_{\cH_V}+CV^*CP_{\cH_V^\perp}+CP_{\cH_V^\perp}C\cD_{V^*}X^*_0P_{\cH_V^\perp}+\cD_{V^*}\wt Y^*\cD_{X^*_0}P_{\cH_V^\perp},\\
C\wt V_{\wt K}^*C=VP_{\cH_V}+C(V^*+\wt K^*\cD_{V^*})CP_{\cH_V^\perp}\\[2mm]
=VP_{\cH_V}+CV^*CP_{\cH_V^\perp}+ C(\wt X_0 + \cD_{X^*_0}\wt YP_{\cL_0^\perp})\cD_{V^*}CP_{\cH_V^\perp}. 
\end{array}
\]
Thus
\begin{equation}\label{reb20a}
\left\{\begin{array}{l}
\wt V_{\wt K}=VP_{\cH_V}+CV^*CP_{\cH_V^\perp}+CP_{\cH_V^\perp}C\cD_{V^*}X^*_0P_{\cH_V^\perp}+\cD_{V^*}\wt Y^*\cD_{X^*_0}P_{\cH_V^\perp},\\[2mm]
C\wt V_{\wt K}^*C=VP_{\cH_V}+CV^*CP_{\cH_V^\perp}+ C(\wt X_0 P_{\cL_0}+ \cD_{X^*_0}\wt YP_{\cL_0^\perp})\cD_{V^*}CP_{\cH_V^\perp}. 
\end{array}\right.
\end{equation}
In particular, taking  \eqref{vfh07a} into account, we derive
\[
\begin{array}{l}
V_{\wt X_0^*}=VP_{\cH_V}+\cD_{V^*}X_0^*P_{\cH_V^\perp}=VP_{\cH_V}+CV^*CP_{\cH_V^\perp}+CP_{\cH_V^\perp}C\cD_{V^*}X^*_0P_{\cH_V^\perp},\\
CV_{\wt X_0^*}C=VP_{\cH_V}+\cD_{V^*}X_0^*P_{\cH_V^\perp}+\left( C\wt X_0\cD_{V^*}C-CP_{\cH_V^\perp}CD_{V^*}X_0^*  \right)P_{\cH_V^\perp}\\
=VP_{\cH_V}+CV^*CP_{\cH_V^\perp}+ C\wt X_0 \cD_{V^*}CP_{\cH_V^\perp},
\end{array}
\]
and for the $C$-self-adjoint contractive extension $\wt W_0$ defined by \eqref{vfh07aad} we get the expression \eqref{vfh07abc}.

Further, from \eqref{hvf08aab} and \eqref{reb20a} we derive
\[
\begin{array}{l}
\wt W=\half\left(\wt V_{\wt K}+C\wt V^*_{\wt K}C\right)=VP_{\cH_V}+\half\left(CV^*C+\cD_{V^*}X_0^*+CX_0 P_{\cL_0}\cD_{V^*}C  \right)P_{\cH_V^\perp}\\[2mm]
+\half\left(\cD_{V^*}\wt Y^*\cD_{X^*_0}+ C\cD_{X^*_0}\wt YP_{\cL_0^\perp}\cD_{V^*}C \right) P_{\cH_V^\perp}\\[2mm]
=\half\left(\wt V_{X_0^*}+C\wt V_{X^*_0}^*C\right)+\half\left(\cD_{V^*}\wt Y^*\cD_{X^*_0}+ C\cD_{X^*_0}\wt YP_{\cL_0^\perp}\cD_{V^*}C \right)
P_{\cH_V^\perp}.
\end{array}
\]
If
\[
M(\wt Y):=\cD_{V^*}\wt Y\cD_{X^*_0},\;\; \wt Y\in\bB(\sD_{X^*_0}, \cD^{-1}_{V^*}(C\cH_V^\perp))\quad\mbox{is a contraction},
\]
then
\[
\wt W=\wt W_0+\half\left(M(\wt Y)+CM^*(\wt Y)C\right)P_{\cH_V^\perp}.
\]
Thus, \eqref{vfh07cda} holds.
\end{proof}

In the special case of codimension one we have the following stronger result.
\begin{corollary}
Retain the assumptions and the notation of Theorem \ref{abe17a}. If $\codim\cH_V=1$, then for each contractive extension
$\wt X\in\bB(\sD_{V^*},\cH_V^\perp)$ of the operator $X_0$ the contractive extension $\wt V_{\wt X^*}$ of the operator $V$ is $C$-self-adjoint;
\end{corollary}
\begin{proof}
We verify  the equality \eqref{abe18bc} for $\wt K=\wt X^*$. Since $\codim\cH_V=1$ by assumption,  $\cH_V^\perp=\dC e_0$, where $e_0\in \cH_V^\perp$, $||e_0||=1.$  Then we
compute
\[
\begin{array}{l}
\wt K^* \cD_{V^*}Ce_0=(\wt K^* \cD_{V^*}Ce_0,e_0)e_0
=(Ce_0,\cD_{V^*}\wt Ke_0)e_0=(C\cD_{V^*}\wt Ke_0, e_0)e_0\\
=(P_{\cH_V^\perp} C\cD_{V^*}\wt Ke_0,e_0)e_0=P_{\cH_V^\perp} C\cD_{V^*}\wt Ke_0.
\end{array},
\]
This implies that \eqref{abe18bc} holds. By Proposition \eqref{abe18aa}, the operator $\wt V_{\wt X^*}$ is a $C$-self-adjoint contractive extension of $V.$
\end{proof}

\begin{remark}\label{hvf17lm}
Since $\wt W_0$ is a contractive extension of $V$, it has the form
\[
\wt W_0=VP_{\cH_V}+\cD_{V^*}\wt K_0P_{\cH_V^\perp}
\]
and $\wt K^*_0\in\bB(\sD_{V^*},\cH_V^\perp)$ is the contractive extension of $X_0$. By \eqref{hvf06a},
\[
\wt K^*_0=X_0P_{\cL_0}+ \cD_{X^*_0}\wt Y_0P_{\cL_0^\perp},\; \quad\mbox{where} \quad\wt Y_0\in \bB(\cL_0^\perp,\sD_{X^*_0})\quad\mbox{is a contraction}.
\]
Since
\[
\wt K_0=X^*_0+\wt Y^*_0\cD_{X^*_0},
\]
it follows from \eqref{vfh07abc} and from the equality $CP_{\cH_V^\perp}C=P_{C\cH_V^\perp}$ (see \eqref{hvf19a})  that
\[
\cD_{V^*}\wt Y^*_0\cD_{X^*_0}f=
\half\left(C\wt X_0\cD_{V^*}C-P_{C\cH_V^\perp}\cD_{V^*}X_0^*  \right)f\;\;\forall f\in\cH_V^\perp.
\]
In particular
\begin{equation}\label{kbg08bb}
f\in\ker\cD_{X^*_0}\Longrightarrow \wt X_0\cD_{V^*}Cf=P_{\cH_V^\perp}C\cD_{V^*}X^*_{0}f.
\end{equation}
\end{remark}

\begin{remark}\label{cherv08a}
Let $\{V,U\}$ be an adjoint pair of non-densely defined contractions. Recall that a contraction $\wt V$ is called a contractive extension of $\{V,U\}$ if $\wt V\supset V$ and $\wt V^*\supset U$.

In the case $\{V,V\}$ ($\Longleftrightarrow$ $V$ is a symmetric contraction) a description of all self-adjoint contractive extensions was obtained by Kre\u{\i}n in \cite{Kr} in terms of an operator interval (see also \cite{KO1977}). The case of all contractive extensions of $\{V,V\}$ has been considered in \cite{ArlTsek1986}.
The existence of a contractive extension for an adjoint pair of non-densely defined contractions (the Cayley transform of an adjoint pair of densely defined accretive operators)
is proved in \cite[Chapter IV, Proposition 4.2]{SF}.

The set of all contractive extensions of an arbitrary adjoint pair of contractions forms an operator ball
\begin{equation}\label{cherv07a}
\wt V_Y=\wt V_0+\cR_l Y\cR_r,
\end{equation}
where $\wt V_0$ is the center and $\cR_l$ and $\cR_r$ are the left and right radii of the operator ball and $Y\in\bB(\cran \cR_r,\cran\cR_l)$ is an arbitrary contraction.
Parametrizations of the ball \eqref{cherv07a} by means of $2\times 2$ block-operator matrices can be found in \cite{AG1982, DKW,FoFra1984,ShYa}.
In other terms the center $\wt V_0$ and the radii $\cR_l$ and $\cR_r$ are described in \cite{Arl1985}. From \eqref{cherv07a} it follows that an adjoint pair of contractions has a unique contractive extension if and only if one of the radii is equal to zero.

Let $V$ be a $C$-symmetric  contraction. Then $\{V,CVC\}$ is the adjoint pair of $C$-symmetric non-densely defined contractions. Moreover, $\wt V$ is a $C$-self-adjoint contractive extension of a $C$-symmetric contraction $V$ if and only if
 $\wt U:=\wt V^*$  is a $C$-self-adjoint extension of $CVC$.
 For the adjoint pair $\{V,CVC\}$ it follows from the results in \cite{Arl1985}  that:
\begin{enumerate}
\item $\wt V_0=\wt V_{X^*_0},$ where the operator $X_0$ is defined in \eqref{abe17b},
\item the operators $\cR^2_l$ and $\cR^2_r$ are the Kre\u{\i}n \textit{shorted operators} \cite{Kr}, \cite{And}:
\[
\begin{array}{l}
\cR^2_l=(\cD^2_{V^*})_{\Omega_l}:=\cD_{V^*}P_{\Omega_l}\cD_{V^*},\\[2mm]
 \cR^2_r=(C\cD^2_{V^*}C)_{\Omega_r}:=C\cD_{V^*}C
 P_{\Omega_r}C\cD_{V^*}C=C\cD_{V^*}P_{\Omega_l}\cD_{V^*}C=C\cR^2_l C,
\end{array}
\]
where the subspaces $\Omega_l$ and $\Omega_r$ are defined as follows (cf. \eqref{cherv07b}):
\[ 
\Omega_l=\{h: \cD_{V^*}h\in CH_V^\perp\},\; \Omega_r=\{g: \cD_{V^*}Cg\in CH_V^\perp\} =C\Omega_l.
\] 
\end{enumerate}
\end{remark}

If $\cL_0^\perp\ne\{0\}\Longleftrightarrow \ran \cD_{V^*}\cap C\cH_V^\perp
\ne\{0\}$, then \eqref{vfh07cda} shows that the $C$-symmetric non-densely defined contraction $V$ has infinitely many $C$-self-adjoint contractive
extensions.

The next theorem deals with the case when there is a  unique extension.
\begin{theorem}\label{kkvit23ab}
If
\begin{equation}\label{abe23b}
 \ran \cD_{V^*}\cap C\cH_V^\perp=\{0\},
\end{equation}
then
\begin{enumerate}
\item[(a)] $\wt W_0=\wt V_{X^*_0}$,
\item[(b)] $\cL_0=\sD_{V^*}$ and $X_0$ is a co-isometry,
\item[(c)] the equality
\begin{equation}\label{vfh07bb}
P_{\cH_V^\perp}C\cD_{V^*}X_0^*f=\wt X_0\cD_{V^*}Cf\;\;\forall f\in\cH_V^\perp
\end{equation}
holds,
\item[(d)]
$\wt W_0$ is the unique $C$-self-adjoint contractive extension of $V$
and, moreover, $\wt W_0$ is the unique contractive extension of the adjoint pair $\{V,CVC\}$ of contractions.

\end{enumerate}
\end{theorem}
\begin{proof}
Assume that $\ran \cD_{V^*}\cap C\cH_V^\perp=\{0\}$. Then, due to the definition \eqref{cherv07b}, we have $\cL_0^\perp=\{0\}$, i.e., $\overline{\cD_{V^*}C\cH_V}=\sD_{V^*}$ and the operator $\wt
K:=X^*_0\in\bB(\cH_V^\perp,\sD_{V^*})$ is the unique contraction  satisfying condition \eqref{abe18bb}. Therefore the operator
$\wt W_0$ is the unique $C$-self-adjoint contractive extension of $V$. Moreover, due to \eqref{abe22ba} and because $\wt K^*$, $\wt L^*$ and $\wt M^*$ are
contractive extensions of $X_0$, we obtain the equalities $\wt K^*=\wt L^*=\wt M^*=X_0$. Thus
$\wt V_{X^*_0}=C\wt V^*_{ X_0^*}C=\wt W_0,$ i.e., $\wt W_0$ is the unique $C$-self-adjoint contractive extension of $V$.

By Proposition \ref{abe18aa}, the equality \eqref{abe18bb} for $\wt K=X^*_0$ holds, i.e.,  \eqref{vfh07bb} is valid.

Next we prove that under the condition \eqref{abe23b} the operator $X_0\in\bB(\sD_{V^*},\cH_V^\perp)$ is a co-isometry, i.e.,
$X_0^*\in\bB(\cH_V^\perp,\sD_{V^*})$ is an isometry.
Since $\wt V_{X_0^*}$ is a $C$-self-adjoint contraction, the equality \eqref{abe23aa} holds. Hence
$$||\cD_{\wt V^*_{X^*_0}}Cf||^2=||\cD_{\wt V_{X^*_0}}f||^2\;\;\forall f\in\cH. $$
Consequently,
\[
\inf\limits_{h\in\cH_V}||\cD_{\wt V^*_{X^*_0}}Cf-\cD_{\wt V^*_{X^*_0}}Ch||^2=
\inf\limits_{h\in\cH_V}||\cD_{\wt V_{X^*_0}}f-\cD_{\wt V_{X^*_0}}h||^2\;\;\forall f\in\cH.
\]
Condition \eqref{abe23b} is equivalent to the equality (see \cite{Kr})
\[
\inf\limits_{h\in\cH_V}||\cD_{V^*}g-\cD_{V^*}Ch||^2=0\;\;\forall g\in\cH.
\]
From the second equality in \eqref{abe23a} we obtain
\[
||\cD_{\wt V^*_{X^*_0}}g-\cD_{V^*_{X^*_0}}Ch||^2=||\cD_{\wt X^*_0}(\cD_{V^*}g-\cD_{V^*}Ch)||^2.
\]
Therefore,
\[
\inf\limits_{h\in\cH_V}||\cD_{\wt V^*_{X^*_0}}Cf-\cD_{\wt V^*_{X^*_0}}Ch||^2=0\;\;\forall f\in\cH.
\]
Since $Y^*\sD_{Y^*}\subseteq\sD_{Y}$ for an arbitrary contraction $Y$ \cite{SF}, from \eqref{abe23a} we get
\[
\inf\limits_{h\in\cH_V}||\cD_{\wt V_{X^*_0}}f-\cD_{\wt V_{X^*_0}}h||^2=||\cD_{X^*_0}P_{\cH_V^\perp}f||^2\;\;\forall f\in\cH,
\]
Hence we conclude that $\cD_{X^*_0}g=0$ for all $g\in\cH_V^\perp$. Thus, $X_0^*$ is an isometry.

Let us prove that $\wt W_0=V_{X^*_0}$ is the unique contractive extension of the adjoint pair of contractions $\{V,CVC\}$.
Suppose that $\wt V$  is a contractive extension of the adjoint pair $\{V,CVC\}$. Since $\wt V\supset V$ and $\wt V^*\supset CVC$,
we get that $C\wt V^*C\supset V$.  Therefore $\wt W:=\half(\wt V+C\wt V^*C)$ is a $C$-self-adjoint contractive extension of $C$-symmetric contraction $V.$  Hence, $\wt W=\wt W_0=\wt V_{X^*_0}$. By Crandall's Theorem \ref{abe21ba} the operator $\wt V$ is of the form $\wt V=\wt V_{\wt K}=VP_{\cH_V}+\cD_{V^*}\wt KP_{\cH_V^\perp}$. Theorem \ref{abe17a} says that $\wt K^*\in \bB(\sD_{V^*}, \cH_V^\perp)$ is a contractive extension of the operator $X_0$. Because $\cL_0=\sD_{V^*}$ and $X_0^*$ is an isometry, the equality
$\wt K=X^*_0$ holds. Thus $\wt V=\wt V_{\wt K}=\wt V_{X^*_0}$. As proved above, $\wt V_{X^*_0}$ is a $C$-self-adjoint contractive extension of $V.$
Thus, the adjoint pair $\{V,CVC\}$ admits a unique contractive extension.

Now the proof of the theorem is complete.
\end{proof}
\begin{remark}\label{cherv10aa}
Remark \ref{cherv08a} and the equality $\ran \cD_{V^*}\cap C\cH_V^\perp=\{0\}$ (see \eqref{abe23b}) yield that $\cR_l=\cR_r=0$. Hence $\wt V_{X^*_0}$ is the unique contractive extension of the adjoint pair $\{V,CVC\}$.
\end{remark}

\begin{remark}\label{hvf24a}
Let $B$ be a bounded nonnegative self-adjoint operator in the Hilbert space $\sH$ and let $\sN$ be a subspace in $\sH$. Then it is known (see e.g.
\cite{Ando1970, Kr, KO1977, Schmudgen}) that the following statements are equivalent:
\begin{enumerate}
\def\labelenumi{\rm (\roman{enumi})}
\item $\ran B^\half\cap\sN=\{0\}$;
\item $\inf\limits_{\f\in\sN^\perp}(B(h-\f),h-\f)=0$ $\forall h\in\sN\setminus\{0\}$;
\item $\max\left\{Z=Z^*\in\bB(\sH): 0\le Z\le B,\;\ran Z\subseteq\sN\right\}=0$;
\item $\sup\limits_{f\in\sH}\cfrac{|(f,h)|^2}{(Bf,f)}=0$ $\forall h\in\sN\setminus\{0\}$ (here the convention $0/0=0$ is used);
\item $\lim\limits_{x\uparrow 0}((B-xI)^{-1} h,h)=+\infty$ $\forall h\in\sN\setminus\{0\}$.
\end{enumerate}
\end{remark}

\section{$C$-self-adjoint unitary extensions of a $C$-symmetric non-densely defined isometric operator}
In this section we consider a particular case of a $C$-symmetric non-densely defined contraction.

Suppose that $V$ is a $C$-symmetric closed non-densely defined isometric operator in the Hilbert space $\cH$, $\dom V=\cH_V$.  By Theorem \ref{abe17a} the operator $V$ admits $C$-self-adjoint contractive extensions.
We are interested in $C$-self-adjoint {\it unitary} extensions of $V$. As it has been mentioned in Remark \ref{cherv08a}, $C$-self adjoint unitary extensions
are unitary extensions of the adjoint pair of isometric operators $\{V, CVC\}$.
Note that when $V$ is an isometry, then for the subspace $\sD_{V^*}$ and for the defect operator $\cD_{V^*}$ the equalities
\[
\sD_{V^*}=\ker V^*=\cH\ominus\ran V,\; \cD_{V^*}=P_{\ker V^*},
\]
hold.

At first, we establish some properties of the closure $X_0$ of the operator $\dot{X_0}$  defined in \eqref{abe17b} and of the corresponding  decompositions of the subspaces $\ker V^*$ and $\cH_V^\perp$.
As  noted in Section \ref{kbg08a}, if $V$ is a $C$-symmetric isometry, then the operator $X_0$ is an isometry defined on the subspace
$$\cL_0=\dom X_0=\overline{\cD_{V^*}C\cH_V}=\overline{P_{\ker V^*}C\cH_V}\subseteq\ker V^*
$$
and $\ran X_0=\overline{P_{\cH_V^\perp}C\ran V}\subseteq\cH_V^\perp$.
\begin{proposition}\label{lip07a}
Let $V$ be a $C$-symmetric isometric operator defined on the subspace $\cH_V$. Then the operator $X_0$ is isometric and
\begin{equation}\label{iul09a}
\sD_{X^*_0}=\ker X_0^*= \cH_V^\perp\cap C\ker V^*.
\end{equation}
For the subspace $\cL_0^\perp=\sD_{V^*}\ominus \dom X_0=\ker V^*\ominus\cL_0$ (see \eqref{cherv07b})
we have the equalities
\begin{equation}\label{lip09aa}
\cL_0^\perp=C\cH_V^\perp\cap\ker V^*=C\ker X_0^*.
\end{equation}
Therefore, the orthogonal decompositions
\begin{equation}\label{lip08cc}
\cH_V^\perp=\ran X_0\oplus\ker X^*_0
\end{equation}
and
\begin{equation}\label{lip08ca}
\ker V^*=\dom X_0\oplus C\ker X^*_0
\end{equation}
hold.
\end{proposition}
\begin{proof}
If $f\in \cH_V^\perp\cap \ker X^*_0$, then  $(f, CVh)=0$ for all $h\in\cH_V$. Hence $Cf\in\ker V^*$ and, therefore, \eqref{iul09a} holds.

Clearly,
\[
\f\in\cL_0^\perp\Longleftrightarrow \f\in\ker V^* \;\mbox{and}\; (\f, Ch)=0\;\;\forall h\in\cH_V\Longleftrightarrow \f\in C\cH_V^\perp\cap\ker V^*=C\ker X^*_0.
\]
\end{proof}
Note that from \eqref{lip08cc} and \eqref{lip08ca} it follows that $\codim (\dom V)=\codim(\ran V).$ Therefore, a $C$-symmetric isometric operator admits always unitary extensions. The next theorem describes those unitary extensions which are $C$-self-adjoint.

\begin{theorem}\label{kbg08ab}
{\rm (1)} The formula
\begin{equation}\label{lip09cbac}
\wt V_J=VP_{\cH_V}+(X^*_0+CJP_{\ker X^*_0})P_{\cH_V^\perp}
\end{equation}
gives
a bijective correspondence between all conjugations $J$ in the subspace $\ker X^*_0(=\cH_V^\perp\cap C\ker V^*)$ and all $C$-self-adjoint unitary extensions $\wt V_J$ of the $C$-symmetric isometric operator $V$.

{\rm (2)} If
\begin{equation}\label{jul10aa}
C\cH_V^\perp\cap\ker V^*=\{0\},
\end{equation}
then the unitary operator $\wt V=VP_{\cH_V}+X^*_0P_{\cH_V^\perp}$ is $C$-self-adjoint and  $\wt V$ is the unique contractive extensions of the adjoint pair $\{V,CVC\}$ of isometric operators.
\end{theorem}
\begin{proof}
 Let $J$ be a conjugation in $\ker X^*_0$ and let  $\wt K_J\in\bB(\cH_V^\perp,\sD_{V^*})$ be the operator defined by
\[
\wt K_J=X^*_0+CJP_{\ker X^*_0} 
\]
From \eqref{lip08cc} and since $X^*_0:\ran X_0\to\dom X_0=\cL_0$ is isometry it follows that
\[
||\wt K_J f||^2=||X^*_0P_{\ran X_0}f||^2+||CJP_{\ker X^*_0}f||^2=||f||^2\;\;\forall f\in\cH_V^\perp.
\]
Because the operator $C$ maps $\cL^\perp_0=C\ker X^*_0$ onto $\ker X^*_0$ for $\wt K^*_J\in\bB(\sD_{V^*}, \cH_V^\perp)$ we have
\[
\wt K^*_J=X_0P_{\cL_0}+JC P_{\cL^\perp_0}
\]
and $||\wt K^*_J g||^2=||g||^2$ for all $g\in\sD_{V^*}=\ker V^*$.
Thus, the operator $\wt K_J$  maps $\cH_V^\perp$ unitarily onto $\sD_{V^*}=\ker V^*$ and the equalities in \eqref{abe18a} and \eqref{abe23a} imply that $\wt V_J=\wt V_{\wt K_J}$ is a unitary extension of the operator $V$.

In order to check that $\wt V_J$ is a C-self-adjoint extension we apply Proposition \ref{abe18aa}. Since $\wt K_J^*$ is a contractive extension of the operator $X_0$, the equality \eqref{abe18bb} is fulfilled.

Taking into account that $\ker \cD_{X^*_0}=\ran X_0$ and $\cD_{V^*}=P_{\ker V^*}$, it follows from \eqref{kbg08bb}
that
\[
\wt X_0P_{\ker V^*}Cf=P_{\cH_V^\perp} CP_{\ker V^*}X^*_0f\;\;\forall f\in\ran X_0.
\]
From the orthogonal decomposition \eqref{lip08ca} and the fact that $X_0$ is isometric we obtain that $X^*_0f\in\dom X_0=\cL_0\subset\ker V^*$ for all $f\in\ran X_0$. So
\[
\wt X_0P_{\ker V^*}Cf=P_{\cH_V^\perp} CX^*_0f\;\;\forall f\in\ran X_0.
\]
If $h\in\ker X^*_0$, then because of $Ch\in C\ker X^*_0=\cL^\perp_0\subset\ker V^*$ (see \eqref{lip09aa}, \eqref{lip08ca}) we get
$$JCP_{\cL^\perp_0}P_{\ker V^*}Ch=Jh$$
and
\[
P_{\cH_V^\perp} C CJP_{\ker X^*_0}h=P_{\cH_V^\perp}Jh=Jh.
\]
Thus, since $\cH_V^\perp\ni g=P_{\ran X_0}g+P_{\ker X^*_0}g$,  we finally obtain
\[
\begin{array}{l}
P_{\cH_V^\perp}CP_{\ker V^*}\wt K_Jg=P_{\cH_V^\perp}C(X^*_0+CJP_{\ker X^*_0})g,\\
=(X_0P_{\cL_0}+JCP_{\cL^\perp_0})P_{\ker V^*}Cg=\wt K^*_JP_{\ker V^*}Cg\;\;\forall g\in\cH_V^\perp.
\end{array}
\]
Now we conclude that \eqref{abe18bc} is fulfilled. Therefore, by Proposition \ref{abe18aa}, the operator $\wt V_J$ is a $C$-self-adjoint extension of $V$.

Suppose now that $\wt V_{\wt K}=VP_{\cH_V}+\wt KP_{\cH_V^\perp}$ is a $C$-self-adjoint unitary extension of $V$. Then $\wt K\in\bB(\cH_V^\perp,\ker V^*)$
is a unitary operator and by Proposition \ref{abe18aa} the operator $\wt K^*\in\bB(\ker V^*,\cH_V^\perp)$ is a contractive extension of $X_0$. Moreover,
$\wt K^*$ is a unitary from $\ker V^*$ onto $\cH_V^\perp$. Then
\[
\wt K^*=X_0P_{\cL_0}+ \wt YP_{\cL_0^\perp},
\]
where $\wt Y$ is a unitary operator from $\cL_0^\perp$ onto $\ker X^*_0$ (see  \eqref{lip09aa} and the orthogonal decompositions
\eqref{lip08ca} and \eqref{lip08cc}). Then $\wt K=X_0^*+\wt Y^* P_{\ker X^*_0}$ and  by Proposition \ref{abe18aa} the equality \eqref{abe18bc} holds. Hence
\[
P_{\cH_V^\perp}C(X_0^*+\wt Y^* P_{\ker X^*_0})f=(X_0P_{\cL_0}+ \wt YP_{\cL_0^\perp})P_{\ker V^*} Cf\;\;\forall f\in\cH_V^\perp.
\]
Then, using \eqref{lip08ca} we get the equality
\[
C \wt Y^*h=\wt YCh\;\; \forall h\in\ker X^*_0.
\]
Define $J:=C\wt Y^*$. Then $J$ is anti-unitary in $\ker X^*_0$ and $J^2h=C\wt Y^*C\wt Y^*h=\wt Y\wt Y^*h=h$ for all $h\in\ker X^*_0$ (see also Lemma \ref{lip09ad}). Therefore
 $\wt Y^*=CJ$ and $\wt V_{\wt K}$ is of the form \eqref{lip09cbac}.

This proves the statement (1).

Since $\ran\cD_{V^*}=\ker V^*$, the condition \eqref{abe23b} turns into the condition \eqref{jul10aa}. In this case,
$\cL_0=\ker V^*$, $\ker X_0^*=\{0\}.$
 Consequently, due to Theorem \ref{kkvit23ab},  (2) is valid.

The proof is complete.
\end{proof}

\begin{remark} \label{luj10aa}
If\, $\ker X^*_0(=\cH_V^\perp\cap C\ker V^*)\ne\{0\}$, then for each conjugation $J$ in $\ker X^*_0$
the $C$-self-adjoint unitary operator $\wt V_J$ of the form \eqref{lip09cbac} is an extension of the $C$-symmetric isometric operator
\[
\dom \wh V_0=\cH_V\oplus\ran X_0,\; \wh V_0(h+g):=V_0h+X^*_0g\;\;\forall h\in\cH_V,\;\forall g\in\ran X_0.
\]
If $CVh=V^*Ch$ for all $h\in\cH_V$, then by Proposition \ref{bkm29a} $\dom X_0=\{0\}$, and, therefore $\wh V_0=V$.
\end{remark}

\section{Dissipative $C$-symmetric operators and their $C$-self-adjoint maximal dissipative extensions}
First we develop some basics on the Cayley transform of a dissipative operator.

\subsection{Dissipative operators and their Cayley transforms}
Recall that
a linear operator $T$ on a Hilbert space $\cH$ is called \textit{dissipative} if the quadratic form
\[
 \gamma_T[f]:=\IM (Tf,f),\; \dom\gamma=\dom T,
\]
is nonnegative.
The operator $T$ is called \textit{maximal dissipative} if it is dissipative and has no dissipative extensions on $\cH$. It is well known  (see e.g.
\cite{Glazman3, Ka,Phillips1959, Straus1968}) that
a closed densely defined operator $T$ is maximal dissipative if and only if $-T^*$ is maximal dissipative ($T^*$ is maximal accumulative).
 If $T$ is a maximal dissipative (\textit{accumulative}) operator, then  the resolvent set of $T$ contains the open lower half-plane $\dC_-$ and the
 resolvent satisfies
$
||(T-z I)^{-1}||\le |\IM z|^{-1}\;\;\forall z,\;\;\IM z<0.
$

The Cayley transform $V_\lambda$ of a dissipative operator $T$ is defined as follows:
\begin{equation}\label{rtkbnh}
\left\{\begin{array}{l} \psi=(T-\bar \lambda I )f\\
V_\lambda \psi=(T-\lambda I)f\end{array}\right. f\in\dom T,\; \IM \lambda>0.
\end{equation}
Then $\dom V_\lambda=\ran (T-\bar\lambda I)$, $\ker (I-V_\lambda)=\{0\}$.
The operator $V_\lambda$ is a contraction. Further, $\dom V_\lambda$ is a closed subspace if and only if the operator $T$ is closed, and $\dom V_\lambda=\cH$
if and only if the operator $T$ is maximal dissipative.

The inverse transform is of the form:
\[
\begin{array}{l}
\left\{\begin{array}{l}f=(I-V_\lambda)h\\
Tf=(\lambda I-\bar\lambda V_\lambda)h\end{array}\right.,\; h\in\dom V_\lambda.
\end{array}
\]
If $T$ is a maximal dissipative operator, then
\[
V_\lambda^*=(T^*-\bar\lambda I)(T^*-\lambda I)^{-1}.
\]
Let $T$ be a closed dissipative operator. Set
\[
\sM_{\bar\lambda}:=\ran (T-\bar\lambda I),\; \;\sN_\lambda=\cH\ominus\sM_{\bar\lambda}.
\]
Then $\sM_{\bar\lambda}$ is a closed subspace of $\cH$ for any $\lambda\in\dC_+$. If $\dom T$ is dense in $\cH$, then $\sN_\lambda=\ker(T^*-\lambda I).$
\begin{proposition}\label{vjtyjdp}
Let $T$ be a maximal dissipative operator and let $V_\lambda =(T-\lambda I)(T-\bar\lambda I)^{-1}$, $\lambda\in\dC_+,$ be its Cayley transform  (see
\eqref{rtkbnh}).
 Then
\[
\begin{array}{l}
||\cD_{\cV_\lambda}\psi||^2
=4~ (\IM \lambda)\,\IM (T(T-\bar\lambda I)^{-1}\psi,(T-\bar\lambda I)^{-1}\psi)
    \quad\forall \psi\in\cH,      \\[2mm]
||\cD_{\cV^*_\lambda}\f||^2
=-4~ (\IM \lambda)\,\IM (T^*(T^*-\lambda I)^{-1}\f,(
T^*-\lambda I)^{-1}\f)
   \quad\forall \f\in\cH.
\end{array}
\]
\end{proposition}
\begin{proposition} \label{reb12aa}
If $T$ is a closed densely defined dissipative operator, then:
\begin{enumerate}
\item[(1)] $\dom T\cap\sN_\lambda=\{0\}\;\; \forall \lambda\in\dC_+;$
\item[(2)]  the operator
\begin{equation}\label{reb12ab}
\left\{\begin{array}{l}
\dom \wt T_\lambda=\dom T\dot+\sN_\lambda\\
\wt T_\lambda(f_T+\f_\lambda)=Tf_T+\lambda\f_\lambda,\; f_T\in\dom T,\f_\lambda\in\sN_\lambda
\end{array}\right.
\end{equation}
is a maximal dissipative extension of $T$ for each $\lambda\in\dC_+$ and
\begin{equation}\label{ber13aa}
(\wt T_\lambda-\lambda I)(\wt T_\lambda-\bar\lambda I)^{-1}=V_\lambda P_{\sM_{\bar\lambda}}.
\end{equation}
\end{enumerate}
\end{proposition}
\begin{proof}
Let $V_\lambda$ be the Cayley transform given by \eqref{rtkbnh}. Then $\dom V_\lambda=\sM_{\bar\lambda}$, $(\dom V_\lambda)^\perp=\sN_\lambda$.

We show that $\ker(I-P_{\sM_{\bar\lambda}}V_\lambda)=\{0\}$.
Suppose that $P_{\sM_{\bar\lambda}}V_\lambda h=h$. Since $V_\lambda$ is a contraction,  $P_{\sN_\lambda}V_\lambda h=0$, i.e., $V_\lambda h=h$. Since
$\ker (I-V_\lambda)=\{0\}$, we get $h=0$.

If $(I-V_\lambda)h\in\sN_\lambda$, then $ (I-P_{\sM_{\bar\lambda}}V_\lambda) h=0$. Hence $h=0$. Because $\ran (I-V_\lambda)=\dom T$, we obtain
$\dom T\cap\sN_\lambda=\{0\}.$

Let  $\wt T_\lambda$ be the operator defined in \eqref{reb12ab}. Then
\[
\begin{array}{l}
(\wt T_\lambda(f_T+\f_\lambda),f_T+\f_\lambda)=(Tf_T+\lambda\f_\lambda,f_T+\f_\lambda)=( T f_T,f_T)+\lambda(\f_\lambda,f_T)\\[2mm]
+(f_T,\lambda\f_\lambda)+\lambda||\f_\lambda||^2.
\end{array}
\]
Hence
\[
\IM(\wt T_\lambda(f_T+\f_\lambda),f_T+\f_\lambda)=\IM (Tf_T,f_T)+\IM\lambda\,||\f_\lambda||^2\ge 0,\;\;f_T\in\dom T,\f_\lambda\in\sN_\lambda.
\]
Thus we have shown that the operator $\wt T_\lambda$ is dissipative. Moreover, the equality
\[
(\wt T_\lambda-\bar\lambda I)(f_T+\f_\lambda)=(T-\bar\lambda I)f_T+(\lambda-\bar\lambda)\f_\lambda
\]
implies that $\ran (\wt T_\lambda-\bar\lambda I)=\sM_{\bar\lambda}\oplus\sN_\lambda=\cH$.
This means that $\wt T_\lambda$ is  maximal dissipative.

The equality \eqref{ber13aa} follows from \eqref{reb12ab} and \eqref{rtkbnh}.
\end{proof}

\begin{corollary}\label{reb13ba}
Let $T$ be a closed densely defined dissipative operator. Then the domain $\dom T^*$ for each $\lambda\in\dC_+$ admits the direct decomposition
\begin{equation}\label{reb14a}
\dom T^*=\dom \wt T^*_\lambda\dot+\sN_\lambda,
\end{equation}
where $\wt T^*_\lambda$ is the adjoint of the maximal dissipative extension $\wt T_\lambda$, defined in \eqref{reb12ab}.
\end{corollary}
\begin{proof}
Since $\ran(\wt T^*-\lambda I)=\cH$, for each $f\in\dom T^*$ there exists a vector $g\in\dom \wt T^*$ such that  $(T^*-\lambda I)f=(\wt T^*-\lambda I)g$. So,
$f=g+\f_\lambda$, where $\f_\lambda\in\sN_\lambda=\ker (T^*-\lambda I)$, and \eqref{reb14a} is  satisfied.
\end{proof}

\subsection{C-symmetric dissipative operators and the uniqueness of $C$-self-adjoint maximal dissipative extensions}

Let $T$ be a closed densely defined $C$-symmetric dissipative operator, i.e., $\IM (Tf,f)\ge 0$ for all $f\in\dom T$ and $T\subseteq CT^*C$.

 According to \cite[Theorem 1]{Glazman1}, \cite[Theorem 1.1]{ELZ_1983}, a closed $C$-symmetric operator $T$ is maximal dissipative if and only if $T$ is
 C-self-adjoint and dissipative.

Define an operator $S$ by
\[
\dom S=C\dom T,\;\;S:=CTC.
\]
Then $S$ is a closed densely defined and anti-dissipative (accumulative) $C$-symmetric operator. Moreover, $(Tf,g)=(f,Sg)$ for all $f\in\dom T$ and all $g\in\dom S.$ Thus $\{T,S\}$ is an adjoint pair. Moreover, $\wt T$ is a $C$-self-adjoint maximal dissipative extension of $T$ if and only if $\wt S:=\wt T^*$ is a $C$-self-adjoint  maximal anti-dissipative extension of $S$.

From \eqref{rtkbnh} it follows that if $T$ is $C$-symmetric ($C$-self-adjoint), then its Cayley transform $V_\lambda$ is $C$-symmetric ($C$-self-adjoint) for
each $\lambda\in\dC_+$.
By Theorem \ref{abe17a}, the Cayley transform $V_\lambda$ admits a $C$-self-adjoint contractive extension $\wt V$. Then $\wt T=(\lambda I-\bar\lambda \wt V)(I-\wt V)^{-1}$ is a $C$-self-adjoint maximal dissipative extension of $T$.
This proves that
 Glazman's Theorem \ref{th1} holds.

Our next aim it to supplement this result by some uniqueness criteria.
Let $\lambda\in\dC_+$ and let $\wt T_\lambda$, see \eqref{reb12ab}, be a maximal dissipative extension of the operator $T$.
If $V_\lambda$ is the Cayley transform of $T$ given by \eqref{rtkbnh}, then $\dom V_\lambda=\sM_{\bar\lambda}$, $V_\lambda$ is a contraction and the Cayley
transform of $\wt T_\lambda$ is the operator $\wt V_\lambda:=V_\lambda P_{\sM_{\bar\lambda}}$, see Proposition \ref{reb12aa} and the equality
\eqref{ber13aa}.

Let us define the following linear space:
\begin{equation}\label{reb22aa}
\cG_\lambda:=(\wt T^*_\lambda-\lambda I)^{-1}C\sM_{\bar\lambda}.
\end{equation}
\begin{theorem}\label{ber22aa}
Suppose that $T$ is a closed $C$-symmetric dissipative operator with dense domain.
 Then the following are equivalent:
 \begin{enumerate}
 \def\labelenumi{\rm (\roman{enumi})}
\item The operator $T$ admits a unique $C$-self-adjoint maximal dissipative extension.
\item The operator $T$ admits a unique maximal dissipative extension $\wt T_0$ such that $\wt T_0^*$ is a maximal anti-dissipative extension of $S:=CTC.$
\item For at least  one $\lambda\in\dC_+$ (and then for all such $\lambda$) the equality
\[
\ran \cD_{V^*_\lambda}\cap C\sN_\lambda=\{0\}
\]
holds.
\item For at least one $\lambda\in\dC_+$ (and then for all such $\lambda$) the equality
\[
\inf\limits_{h\in \sM_{\bar\lambda}}
||\cD_{V^*_\lambda}C \f_\lambda-\cD_{V^*_\lambda} Ch||^2=0
\]
is satisfied for each $\f_\lambda\in \sN_\lambda\setminus\{0\}$.
\item For at least  one $\lambda\in\dC_+$ (and then for all such $\lambda$) the equality
\[
\sup\limits_{f\in \cH}\cfrac{|(Cf,\f_\lambda)|^2}{||\cD_{V^*_\lambda} f||^2}=+\infty
\]
holds for each $\f_\lambda\in\sN_{\lambda}\setminus\{0\}$.
\item For at least  one $\lambda\in\dC_+$ (and then for all such $\lambda$) the equality
\[
\inf\limits_{\psi\in\cG_\lambda}\IM\left(-\wt T^*_\lambda(f-\psi),f-\psi\right)=0\;\;\forall f\in\dom\wt T^*_\lambda\setminus\{0\}
\]
is valid, where $\cG_\lambda$ is defined in \eqref{reb22aa}.
\item For at least one $\lambda\in\dC_+$ (and then for all such $\lambda$) the equality
\[
\sup\limits_{g\in\dom \wt T^*_\lambda} \cfrac{|(C(\wt T^*_\lambda-\bar\lambda I)g,\f_\lambda)|^2}{-\IM (\wt T^*_\lambda g,g)}=+\infty
\]
holds for each $\f_\lambda\in\sN_{\lambda}\setminus\{0\}$.
\end{enumerate}
\end{theorem}
\begin{proof}
Set $\cH_{V_\lambda}:=\sM_{\bar\lambda}$. Then $\cH_{V_\lambda}^\perp=\sN_\lambda$. The non-densely defined contraction $V_\lambda$ is  $C$-symmetric.
By Theorem \ref{abe17a} we have:\,
 \eqref{abe23b} $\Longleftrightarrow$ the operator $V:=V_\lambda$ admits a unique $C$-self-adjoint contractive extension $\wt V_0$ $\Longleftrightarrow$ $\wt V_0$ is the unique contractive extension of the adjoint pair of contractions $\{V_\lambda,CV_\lambda C\}$. But \eqref{abe23b}
is the same as  $\ran\cD_{V_\lambda^*}\cap C\sN_\lambda=\{0\}$.
  This proves the equivalences of (i), (ii), and (iii).

The equivalence of the other statements  can be shown by applying  Remark \ref{hvf24a} with  $B:=\cD^2_{V_\lambda^*}$ and $\sN:=C\sN_\lambda$. We carry out
this for the conditions (v) and (vi) in detail.

Note that
$\dom V_\lambda =\cH_{V_\lambda}=\sM_{\bar \lambda}$, $C\cH_{V_\lambda}=C\sM_{\bar\lambda}$.
By Proposition \ref{vjtyjdp},
\[
||\cD_{V^*_\lambda}\f||^2
=-4\IM \lambda\,\IM (\wt T^*_\lambda(\wt T^*_\lambda-\lambda I)^{-1}\f,(\wt T^*_\lambda-\lambda I)^{-1}\f),\; \f\in\cH.
\]
It follows that for $Ch$, $h\in \sM_{\bar\lambda}$, we have
\[
||\cD_{V^*_\lambda}Ch||^2
=-4\IM \lambda\,\IM (\wt T^*_\lambda(\wt T^*_\lambda-\lambda I)^{-1}Ch,(\wt T^*-\lambda I)^{-1}Ch).
\]
Hence
\[
\begin{array}{l}
 ||\cD_{V^*_\lambda}g-\cD_{V^*_\lambda}Ch||^2\\
 = -4\IM \lambda\,\IM \left(\wt T^*_\lambda((\wt T^*_\lambda-\lambda I)^{-1}g-(\wt T^*_\lambda-\lambda I)^{-1}Ch),(\wt T^*_\lambda-\lambda I)^{-1}g-(\wt
 T^*-\lambda I)^{-1}Ch\right).
\end{array}
\]
Therefore,
\[\begin{array}{l}
\inf\limits_{h\in\sM_{\bar\lambda}}||\cD_{V^*_\lambda}g-\cD_{V^*_\lambda}Ch||^2=0\Longleftrightarrow\\[2mm]
\inf\limits_{h\in\sM_{\bar\lambda}}\Bigl(-\IM \left(\wt T^*_\lambda((\wt T^*_\lambda-\lambda I)^{-1}g-(\wt T^*_\lambda-\lambda I)^{-1}Ch),(\wt
T^*_\lambda-\lambda I)^{-1}g-(\wt T^*-\lambda I)^{-1}Ch\right)\Bigr)=0\\[2mm]
\Longleftrightarrow \inf\limits_{h\in\sM_{\bar\lambda}}\Bigl(-\IM \left(\wt T^*_\lambda((\wt T^*_\lambda-\lambda I)^{-1}g-(\wt T^*_\lambda-\lambda
I)^{-1}Ch),(\wt T^*_\lambda-\lambda I)^{-1}g-(\wt T^*-\lambda I)^{-1}Ch\right)\Bigr)=0\\[2mm]
\Longleftrightarrow \inf\limits_{\psi\in\cG_\lambda}\Bigl(-\IM \left(\wt T^*_\lambda(f-\psi),f-\psi\right)\Bigr)=0,\;\; f=(\wt T^*_\lambda-\lambda I)^{-1}g.
\end{array}
\]
Since (see Remark \ref{hvf24a})
\[
\inf\limits_{h\in\sM_{\bar\lambda}}||\cD_{V^*_\lambda}g-\cD_{V^*_\lambda}Ch||^2=0\;\;\forall g\in\cH\Longleftrightarrow
\ran \cD_{V^*_\lambda}\cap C\sN_\lambda=\{0\},
\]
we get
\[
\inf\limits_{\psi\in\cG_\lambda}\IM\left(-\wt T^*_\lambda(f-\psi),f-\psi\right)=0\;\;\forall f\in\dom\wt T^*_\lambda\setminus\{0\}.
\]
Thus (i) $\Longleftrightarrow$ (ii) $\Longleftrightarrow$ (v).

Further,
\[
\begin{array}{l}
\cfrac{|(Cf,h)|^2}{||\cD_{V^*_\lambda} f||^2}=-\cfrac{|(Cf,h)|^2}{4\IM \lambda\,\IM (\wt T^*_\lambda(\wt T^*_\lambda-\lambda I)^{-1}f,(\wt
T^*_\lambda-\lambda I)^{-1}f)}\\[2mm]
\quad=-\cfrac{|(C(T^*_\lambda-\lambda I)g,h)|^2}{4\IM \lambda\,\IM (\wt T^*_\lambda g,g)},\; g=(T^*_\lambda-\lambda I)^{-1}f.
\end{array}
\]
Since (v) $\Longleftrightarrow$ (iii), we conclude that  (vii) $\Longleftrightarrow$ (iii).
\end{proof}

Let $V:=(T-iI)(T+iI)^{-1}$ be the Cayley transform of a $C$-symmetric closed densely defined operator $T$.
Then $\cH_V:=\dom V=\sM_{-i}$, $\cH_V^\perp=\sN_i$ and the operator $X_0$ defined by \eqref{abe17b} takes the form
\[
X_0:\cD_{V^*}Ch\mapsto P_{\sN_i}CVh,\; h\in\sM_{-i}.
\]
Let $X^*_0:\sN_{i}\mapsto \overline{\cD_{V^*}C\sM_{-i}}$ be its adjoint operator.
We define, see \eqref{abe18a},
\[
\wt V_{X^*_0}=VP_{\sM_{-i}}+\cD_{V^*}X^*_0P_{\sN_i}.
\]
Then, using Theorem \ref{abe23b}, we describe the unique $C$-self-adjoint maximal dissipative extension in the following theorem.
\begin{proposition}\label{kvit30bab} Supose that $T$ is a $C$-symmetric closed densely defined operator $T$ such  that one of the conditions in Theorem
\ref{ber22aa} is fulfilled. Let $V:=(T-iI)(T+iI)^{-1}$ be the Cayley transform of $T$. Then
the unique $C$-self-adjoint maximal dissipative extension of $T$,
 that is, the  unique maximal dissipative extension of the dual pair $\{T, CTC\}$,
 is the inverse Cayley transform of the operator $\wt V_{X^*_0}$, i.e.,
\[
\wt T_0=i(I+\wt V_{X^*_0})(I-\wt V_{X^*_0})^{-1}.
\]
\end{proposition}
A special case of dissipative operators are symmetric operators. Note that each densely defined closed symmetric operator is either maximal symmetric or admits maximal symmetric extensions \cite{AG}, \cite{Schmudgen}.

Let $A$ be a symmetric operator. Then the Cayley transform $V_\lambda=(A-\lambda I)(A-\bar\lambda I)^{-1}$ ($\IM \lambda\ne 0$), is an isometric operator such that\,
$\dom V_\lambda=\sM_{\bar\lambda}$, $(\dom V_\lambda)^\perp=\sN_\lambda$, $\ran V_\lambda=\sM_\lambda$, and $\ker V^*_\lambda=\sN_{\bar\lambda}$.

The next theorem follows immediately  from Proposition \ref{lip07a}, Theorem \ref{kbg08ab} and Theorem \ref{ber22aa}.
\begin{theorem} \label{kbh10aa}
Let $A$ be a densely defined closed and non-maximal symmetric operator. Suppose, in addition, that $A$ is $C$-symmetric w.r.t. a conjugation $C$. Then $A$ admits self-adjoint and $C$-self-adjoint extensions. If the condition $\sN_{\bar\lambda}\cap C\sN_\lambda=\{0\}$ is fulfilled for some nonreal $\lambda$, then it is true for each nonreal $\lambda$ and  the adjoint pair $\{A, CAC\}$ admits a unique maximal dissipative extension. This unique extension is a self-adjoint and $C$-self-adjoint operator.
\end{theorem}
\smallskip

\subsection*{Acknowledgement}
 Yu.~Arlinski\u{\i} thanks  the Deutsche Forschungsgemeinschaft for support from the research grant SCHM1009/6-1.

\end{document}